\documentclass[11pt,a4paper]{amsart}
\usepackage[myheadings]{fullpage}
\usepackage{amssymb}
\usepackage{enumerate}
\usepackage{pinlabel}
\usepackage{graphicx}
\usepackage{caption}
\usepackage{hyperref}
\hypersetup
{
colorlinks=true,   
linkcolor=blue,
citecolor=blue,
}

\theoremstyle{plain}
\newtheorem{cor*}{Corollary}
\newtheorem{prop*}{Proposition}
\newtheorem{theorem}{Theorem}[section]
\newtheorem{cor}{Corollary}[theorem]
\newtheorem{prop}[theorem]{Proposition}
\newtheorem{lem}[theorem]{Lemma}
\newtheorem{algo}[theorem]{Algorithm}

\theoremstyle{definition}
\newtheorem{defn}[theorem]{Definition}
\newtheorem{exmp}[theorem]{Example}
\newtheorem{rem}[theorem]{Remark}

\newcommand{\F}{\mathcal{F}}
\newcommand{\z}{\mathbb{Z}}
\newcommand{\Orb}{\mathcal{O}}

\DeclareMathOperator{\clmap}{CLMod}
\DeclareMathOperator{\aut}{Aut}
\DeclareMathOperator{\out}{Out}
\DeclareMathOperator{\lcm}{lcm}
\DeclareMathOperator{\pmap}{PMod}
\DeclareMathOperator{\lmap}{LMod}
\DeclareMathOperator{\smap}{SMod}
\DeclareMathOperator{\map}{Mod}
\DeclareMathOperator{\homeo}{Homeo^{+}}
\everymath{\displaystyle}

\begin{document}

\title[Liftable mapping class groups of cyclic covers of spheres ]{Generating the liftable mapping class groups of \\cyclic covers of spheres}

\author{Pankaj Kapari}
\address{Department of Mathematics\\
Indian Institute of Science Education and Research Bhopal\\
Bhopal Bypass Road, Bhauri \\
Bhopal 462 066, Madhya Pradesh\\
India}
\email{pankajkapri02@gmail.com}
\urladdr{https://sites.google.com/view/pankajkapdi/home}

\author{Kashyap Rajeevsarathy}
\address{Department of Mathematics\\
Indian Institute of Science Education and Research Bhopal\\
Bhopal Bypass Road, Bhauri \\
Bhopal 462 066, Madhya Pradesh\\
India}
\email{kashyap@iiserb.ac.in}
\urladdr{https://home.iiserb.ac.in/$_{\widetilde{\phantom{n}}}$kashyap/}

\author{Apeksha Sanghi}
\address{Department of Mathematics\\
Indian Institute of Science Education and Research Mohali\\
Knowledge city, Sector 81, SAS Nagar, Manauli PO 140306\\
Punjab, India}
\email{apekshasanghi93@gmail.com}



\begin{abstract}
For $g\geq 2$, let $\text{Mod}(S_g)$ be the mapping class group of closed orientable surface $S_g$ of genus $g$. In this paper, we derive a finite generating set for the liftable mapping class groups corresponding to finite-sheeted regular branched cyclic covers of spheres. As an application, we provide an algorithm to derive presentations of these liftable mapping class groups, and the normalizers and centralizers of periodic mapping classes corresponding to these covers. Furthermore, we determine the isomorphism classes of the normalizers of irreducible periodic mapping classes in $\text{Mod}(S_g)$. Moreover, we derive presentations for the liftable mapping class groups corresponding to covers induced by certain reducible periodic mapping classes. Consequently, we derive a presentation for the centralizer and normalizer of a reducible periodic mapping class in $\text{Mod}(S_g)$ of the highest order $2g+2$. As final applications of our results, we recover the generating sets of the liftable mapping class groups of the hyperelliptic cover obtained by Birman-Hilden and the balanced superelliptic cover obtained by Ghaswala-Winarski.   
\end{abstract}

\maketitle

\section{Introduction}
\label{sec:intro}
For integers $g,k \geq 0$, let $S_{g,k}$ be a connected, closed, and orientable surface of genus $g$ with $k$ marked points (with the understanding that $S_{g,0}$ is denoted by $S_g$), and let $\map(S_{g,k})$ be the mapping class group of $S_{g,k}$. It is a natural question to ask whether one can obtain a presentation for the normalizers and the centralizers of periodic mapping classes. Birman-Hilden obtained \cite[Theorem 8]{birman71} a presentation for the normalizer of a hyperelliptic involution. Furthermore, they showed~\cite[Theorem 4-5]{birman73} that for certain finite-sheeted regular branched covers of hyperbolic surfaces, the normalizers of periodic mapping classes arise as the lifts of the liftable mapping class groups associated with these covers (see Subsection \ref{subsec:bh_theory}). Consequently, they obtained~\cite[Section 5]{birman73} a presentation for the normalizers of periodic mapping classes for covers over the sphere, assuming that every mapping class lifts. Recently, Ghaswala-Winarski~\cite[Theorem 1.3]{winarski17} classified the finite-sheeted regular cyclic branched coverings over the sphere under which every mapping class lifts. In a subsequent work~\cite[Theorem 5.7]{ghaswala17}, they obtained a presentation for the liftable mapping class group of a balanced superelliptic cover. More recently, Hirose-Omori~\cite[Theorem 5.3]{hirose22} obtained another presentation of the liftable mapping class group of the balanced superelliptic cover, hence its normalizer~\cite[Theorem 6.12]{hirose22}. For unbranched regular cyclic covers over non-spherical surfaces of genus atleast three, Dey-Dhanwani-Patil-Rajeevsarathy~\cite[Theorem 2.10]{rajeevsarathy212} derived a finite generating set for the liftable mapping class groups.

As the main result of this paper (see Theorem~\ref{thm:main} and Corollary \ref{cor:main}), we generalize these results by providing a generating set for the liftable mapping class groups associated with arbitrary cyclic branched covers of spheres. As an application, we provide an algorithm (see Algorithm~\ref{algo:main}) for obtaining a presentation for the liftable mapping class groups, the normalizers, and the centralizers of these periodic mapping classes. To prove our results, we have used an algebraic characterization of the Birman-Hilden property due to Broughton~\cite[Theorem 3.2]{broughton92} (see Theorem \ref{thm:norm_cent}) and the theory of geometric realizations of cyclic actions obtained by Parsad-Rajeevsarathy-Sanki~\cite[Theorem 2.7, 2.24]{rajeevsarathy19}.

In \cite[Example 4.1]{broughton92}, Broughton has described the isomorphism classes of the normalizers of irreducible periodic mapping classes of prime orders. First, we generalize this result to the class of all irreducible periodic mapping classes (see Proposition \ref{prop:norm_irr}). A similar result has been obtained for the centralizers of irreducible periodic mapping classes by Kapari-Rajeevsarathy-Sanghi~in \cite[Proposition 4.12]{kapdi24}. Furthermore, we classify the isomorphism types of the normalizers and the centralizers of all irreducible periodic mapping classes (up to conjugacy and power) in $\map(S_3)$ along with the conjugacy classes of their generators (see Table \ref{tab:first}).

For even integer $g>0$, let $F' \in \map(S_{g/2})$ be an irreducible periodic mapping class represented by a $2\pi/n$-rotation $\F'$ of a polygon $\mathcal{P}$ (following the theory developed in \cite{rajeevsarathy19}). We take two copies of $\mathcal{P}$ such that one supports $\F'$ and another supports $\F'^{-1}$. We remove open (invariant) disks about the center of these polygons and glue the resultant boundaries. This gives a reducible periodic mapping class $F \in \map(S_g)$ of order $n$. As an application of Algorithm \ref{algo:main}, we have obtained a presentation for the liftable mapping class group $\lmap_p(S_{0,4})$ of the cover $p$ corresponding to $F$ (see Proposition \ref{prop:present}).

\begin{prop*}
\label{prop:present_lmod}
For integers $g>0$ and $n>g+1$ such that $g$ is even, let $F \in \map(S_g)$ be a periodic mapping class as described above. Then:
\begin{enumerate}[(i)]
\item $\lmap(S_{0,4}) = \left \langle a, b, c \mid a^2=b^2,~[a,b]=(ac)^2=(bc)^2=1 \right \rangle$, if the corresponding orbifold of $F$ has a branch point of order $2$, and
\item $\lmap(S_{0,4}) = \left \langle a, b, c \mid c^2=[a,c]=[b,c]=1\right \rangle$, otherwise, 
\end{enumerate}
where, $a$ and $b$ represent half-twists, and $c$ represents a Dehn twist in $\map(S_{0,4})$.
\end{prop*}

\noindent  As a consequence of the presentation derived in Proposition \ref{prop:present_lmod}, we have derived a presentation for the normalizer $N(F)$ and the centralizer $C(F)$ of $F$ when $n=2g+2$.
\begin{cor*}
\label{cor*:2g+2}
For an even integer $g>0$, let $F \in \map(S_g)$ be a periodic mapping class of order $2g+2$ as described above. Then: 
\begin{enumerate}[(i)]
\item $C(F)=\left\langle F, G_1, G_2 \mid F^{2g+2}=[G_1,F]=[G_2,F]=1, ~(G_1G_2)^2=F^{g+2}\right\rangle
$ and
\item $N(F) =\begin{array}{rl}
 \big\langle F, G_1, G_2, G_3 & \mid F^{2g+2}= [G_1,F]=[G_2,F]=[G_1,G_3]=1,\\
& (G_1G_2)^2=F^{g+2}, G_3FG_3^{-1}=F^{-1}, G_1^2=G_3^2, (G_3G_2)^2=F^{g+1} \big\rangle,
 \end{array}$
\end{enumerate}
where the $G_i$ are pseudo-periodic mapping classes.
\end{cor*}

\noindent Furthermore, for periodic mapping classes described above, when $g+1<n<2g$, we have derived the following presentation for the centralizer $C(F)$ of $F$.

\begin{cor*}
For integers $g>0$ and $g+1<n<2g$ such that $g$ is even, let $F \in \map(S_g)$ be a periodic mapping class as described above. Then:  
$$C(F) = \langle F,G_1,G_2 \mid F^n=[G_1,F]=[G_2,F]=1 \rangle,$$
where the $G_i$ are pseudo-periodic mapping classes.
\end{cor*}

As final applications of Theorem \ref{thm:main}, in Subsection \ref{subsec:lift_bh}, we recover the generating sets of the liftable mapping class groups of a hyperelliptic cover obtained by Birman-Hilden \cite{birman71} and a balanced superelliptic cover obtained by Ghaswala-Winarski \cite{ghaswala17} (see Proposition \ref{prop:b-h} - \ref{prop:g-w}). This paper is organized as follows. In Section~\ref{sec:prelim}, we provide a succinct review of the concepts and results used in this paper. In Section~\ref{sec:liftable}, we derive the main result, and its applications are described in Section~\ref{sec:app}.
  
\section{Preliminaries}
\label{sec:prelim}
For integers $g,k \geq 0$, let $S_{g,k}$ be a connected, closed, and orientable surface of genus $g$ with $k$ marked points. Let $\homeo(S_{g,k})$ be the group of all orientation-preserving self-homeomorphisms of $S_{g,k}$, which preserve the set of marked points. The \textit{mapping class group} of $S_{g,k}$, denoted by $\map(S_{g,k})$, is the group of path components of $\homeo(S_{g,k})$. The elements of $\map(S_{g,k})$ are called \textit{mapping classes}. We will write $S_g$ for $S_{g,k}$ whenever $k=0$. Throughout this article, we assume that $g\geq 2$ whenever $k=0$, unless mentioned otherwise.

\subsection{Cyclic actions on surfaces} A finite cyclic group $G$ is said to \textit{act on $S_g$ via homeomorphisms} if there exist an injective homomorphism $\epsilon: G \to \homeo(S_g)$. Let $\aut(G)$ be the automorphism group of $G$. For $i=1,2$, two finite cyclic group actions $\epsilon_i:G \to \homeo(S_g)$ on $S_g$ are said to be \textit{conjugate} if there exist an $h\in \homeo(S_g)$ and an automorphism $\omega \in \aut(G)$ such that 
$\epsilon_2(x)=h\epsilon_1(\omega(x))h^{-1}$ for every $x\in G$. We always identify $G$ with its image $\epsilon(G)$ in $\homeo(S_g)$. Consider the natural quotient map $\pi: \homeo(S_g)\to \map(S_g)$. By the Nielsen realization theorem \cite{nielsen43,fenchel48,kerckhoff80,kerckhoff83}, any finite cyclic subgroup of $\map(S_g)$ is an isomorphic image of a finite cyclic subgroup of $\homeo(S_g)$ under the projection map $\pi$. The following result (see \cite[Proposition 2.1]{broughton92}) is a consequence of the Nielsen realization theorem and a result of Macbeath \cite[Theorem 3]{macbeath67}.

\begin{prop}
\label{prop:group_action}
For a finite cyclic group $G$ acting on $S_g$, the map $G\to \pi(G)$ is a one-to-one correspondence between conjugacy classes of finite cyclic group actions on $S_g$ and conjugacy classes of finite cyclic subgroups of $\map(S_g)$.
\end{prop}

By the Nielsen realization theorem, a periodic mapping class $F \in \map(S_g)$ is represented by an $\F\in \homeo(S_g)$, known as its \textit{Nielsen representative} of order $|F|$. Furthermore, $\F$ can be chosen as an isometry of some hyperbolic metric on $S_g$. Thus, the $\langle \F \rangle$-action on $S_g$ induces a finite-sheeted regular branched cover $p: S_g \rightarrow \Orb_F$, where $\Orb_F$ denotes the \textit{quotient orbifold} $S_g/\langle \F \rangle$. The cover $p$ has branched points $x_1, x_2, \dots , x_k$ of orders $n_1, n_2, \dots, n_k $, respectively, where each $n_i \mid |F|$. Topologically, $\Orb_F$ is homeomorphic to $S_{g_0,k}$, for some integer $g_0\geq 0$ known as the \textit{orbifold genus} of $\Orb_F$. The tuple $(g_0;n_1,n_2,\dots,n_k)$ is called the \textit{signature} of $\Orb_F$ and will be denoted by $\Gamma(\Orb_F)$. Throughout this article, in a given signature $(g_0;n_1,n_2,\dots,n_k)$, the $n_i$'s are arranged so that $n_1\leq n_2\leq \dots\leq n_k$. By the multiplicativity of the Euler characteristic under the cover $p$, one can deduce the Riemann-Hurwitz formula:
\begin{equation}
\label{eqn:rh}
\frac{2g-2}{n} = 2g_0-2 + \sum_{i=1}^{k} \left( 1- \frac{1}{n_i} \right).
\end{equation}

Since $g \geq 2$, $\Orb_F$ is a hyperbolic orbifold with universal cover isometric to the Poincare upper half plane $\mathbb{H}$. The deck group of the universal cover $\mathbb{H}\to \Orb_F$, denoted by $\pi_1^{orb}(\Orb_F)$, is called the \textit{orbifold fundamental group} of $\Orb_F$. From the orbifold covering space theory \cite[Chapter 13]{thurston}, we have the following short exact sequence
\begin{equation}
\label{eqn:ses}
1 \rightarrow \pi_1(S_g) \rightarrow \pi_1^{orb}(\Orb_F) \xrightarrow{\eta} \langle F \rangle \rightarrow 1.
\end{equation}
The group $\pi_1^{orb}(\Orb_F)$ is a co-compact Fuchsian group (see \cite{macbeath,katok}) and admits a presentation of the following form:
\begin{equation}
\label{eqn:gamma}
\left\langle \alpha_1, \beta_1, \dots , \alpha_{g_0}, \beta_{g_0}, \gamma_1, \dots , \gamma_k ~\Big|~ \gamma_i^{n_i} = \prod_{j=1}^{g_0}[\alpha_j,\beta_j] \prod_{i=1}^{k} \gamma_i = 1 \right\rangle.
\end{equation}
A group having a presentation of the form (\ref{eqn:gamma}) will be denoted by $\Gamma$. The homomorphism $\eta$ from the sequence~(\ref{eqn:ses}) above is known as the \textit{surface kernel map}.

\begin{rem}
We identify the cyclic group $\langle F\rangle$ with $\z_n$ where the generator $F$ is mapped to $1\in \z_n$ under this identification. Furthermore, $\aut(\langle F\rangle)$ will be identified with $\z_n^{\times}$ such that the map $F\to F^{\ell}$ is identified with the map $1 \to \ell$, where $\gcd(\ell,n)=1$. Any surface kernel map associated with $F$ will be written as $\Gamma \to \z_n$.
\end{rem}

For $\eta: \Gamma \rightarrow \z_n$ and $1 \leq i \leq g_0, ~1 \leq j \leq k$, let $\eta(\alpha_i) = a_i$, $\eta(\beta_i) = b_i$, and $\eta(\gamma_j) = c_j$. Since $\eta$ is a surjective homomorphism with a torsion-free kernel, we have
\begin{equation}
\label{eqn:order}
|c_j| = n_j,
\end{equation}

\begin{equation}
\label{eqn:long}
\prod_{j=1}^{g_0}[a_j,b_j]\prod_{j=1}^{k} c_j=\prod_{j=1}^{k} c_j = 1, \text{ and }
\end{equation}

\begin{equation}
\label{eqn:gen}
\left\langle a_1,b_1, \dots , a_{g_0}, b_{g_0}, c_1, c_2, \dots , c_k \right\rangle = \z_n.
\end{equation}
\noindent In classical parlance~\cite{broughton91}, an ordered $(2g_0+k)$-tuple $(a_1,b_1, \dots, a_{g_0}, b_{g_0}, c_1,c_2, \dots ,c_k)$ of elements of $\z_n$ is called a \textit{generating $\Gamma$-vector} if it satisfies Equations (\ref{eqn:order}) - (\ref{eqn:gen}) above. 

\begin{defn}
For $i=1,2$, let $\eta_i: \Gamma \rightarrow \z_n$ be two surface kernel maps. Then we say that \textit{$\eta_1$ is equivalent to $\eta_2$} if there exist $\delta \in \aut(\Gamma)$ and $\ell \in \z_n^{\times}$ such that $(\ell,\delta)\cdot \eta_1:= \ell \circ \eta_1 \circ \delta=\eta_2$, where $\ell$ acts by multiplication.
\end{defn}

\noindent There is a one-to-one correspondence between the set of all surface kernel maps $\Gamma \to \z_n$ and the set of all generating $\Gamma$-vectors \cite[Theorem 3]{harvey66}. Therefore the group $\z_n^{\times} \times \aut(\Gamma)$ acts on the set of all surface kernel maps and consequently acts on the set of all generating $\Gamma$-vectors.

Throughout this article, the symmetric group on $k$ letters will be denoted by $\Sigma_k$. The following result (see~\cite[Theorem 5.8.2]{ZVC80}) describes automorphisms of a co-compact Fuchsian group $\Gamma$. 

\begin{theorem}
\label{thm:aut_gamma}
Let  $c=\textstyle\prod_{j=1}^{g_0}[\alpha_j,\beta_j]\prod_{i=1}^k \gamma_i$ be the element of the free group
\[
\mathbb{F}=\langle \alpha_1,\beta_1,\dots,\alpha_{g_0},\beta_{g_0},\gamma_1,\gamma_2,\dots,\gamma_k \rangle.
\]
Then each automorphism of $\Gamma$ is induced from an automorphism $\phi:\mathbb{F} \to \mathbb{F}$ such that $\phi(c),\phi(\gamma_1)$, $\dots,\phi(\gamma_k)$ are conjugate to $c^s,\gamma_{\sigma(1)}^s,\dots,\gamma_{\sigma(k)}^s$ respectively, where $s=\pm1$ and $\sigma \in \Sigma_k$ such that $n_i=n_{\sigma(i)}$. 
\end{theorem}

\noindent An automorphism of $\Gamma$ is said to \textit{orientation-preserving} if $s=1$ in Theorem \ref{thm:aut_gamma}. Let $\aut^+(\Gamma)$ denotes the index two subgroup of $\aut(\Gamma)$ consisting of orientation-preserving automorphisms. The following result (see  \cite[Proposition 2.2]{broughton91}) is a consequence of the works of Harvey \cite[Theorem 3]{harvey66} and Macbeath \cite[Theorem 3]{macbeath67}. 

\begin{prop}
\label{prop:gen_vector}
There is a one-to-one correspondence between conjugacy classes of finite cyclic group actions on $S_g$ and $\z_n^{\times}\times \aut^+(\Gamma)$-equivalence classes of generating $\Gamma$-vectors. 
\end{prop}

Let $\eta:\Gamma \to \z_n$ be a surface kernel map. For $1 \leq j\leq k$, since $|c_j| = |\gamma_j| = n_j,$ we have $\eta(\gamma_j) = (n/n_j)d_j = c_j$ for some $d_j \in \z_{n_j}^{\times}$. The $\z_n$-action on $S_g$ can be encoded by the tuple
\[
(n,g_0;(d_1,n_1), (d_2,n_2), \dots , (d_k,n_k))
\]
known as a cyclic data set first appeared in \cite[Definition 2.1]{rajeevsarathy19}.

\begin{defn}
\label{defn:dataset}
A tuple of non-negative integers of the form
\[
D=(n,g_0;(d_1,n_1), (d_2,n_2), \dots , (d_k,n_k))
\]
is called a \textit{cyclic data set} of \textit{genus $g$ and degree $n$} if the following conditions hold.
\begin{enumerate}[(i)]
\item For each $1 \leq i \leq k$, $n_i \mid n$, $n_i \geq 2$, and $\gcd(d_i,n_i) = 1$.
\item $\lcm(n_1, \dots , \Hat{n_i}, \dots , n_k) = \lcm(n_1, n_2, \dots, n_k)$ for all $1 \leq i \leq k$.
\item If $g_0 = 0$, then $\lcm(n_1, n_2, \dots, n_k) = n$.
\item $\sum_{i=1}^{k} \frac{n}{n_i}d_i \equiv 0 \pmod{n}$.
\label{eqn:angle_sum}
\item $\frac{2g-2}{n} = 2g_0 -2 + \sum_{i=1}^{k} \left( 1- \frac{1}{n_i} \right).$
\end{enumerate}
\end{defn}

\noindent The integers $n$, $g_0$, and $g$ corresponding to a data set $D$ will be denoted by $n(D)$, $g_0(D)$, and $g(D)$ respectively. The following definition describes the equivalence of two cyclic data sets.

\begin{defn}
Two cyclic data sets
\[
(n,g_0;(d_1,n_1), (d_2,n_2), \dots , (d_k,n_k))
\]
and
\[
(n',g_0';(d_1',n_1'), (d_2',n_2'), \dots , (d_{k'}',n_{k'}'))
\]
are said to be \textit{equivalent} if the following conditions hold.
\begin{enumerate}[(i)]
\item $n=n'$, $g_0 = g_0'$, and $k = k'$.
\item there exists $\ell \in \z_n^{\times}$, and $\sigma \in \Sigma_k$ such that $(\ell d_i,n_i) = (d_{\sigma(i)}', n_{\sigma(i)}')$ for all $1 \leq i \leq k$.
\end{enumerate}
\end{defn}

\noindent The following result originally due to Nielsen \cite{nielsen37} follows from Propositions \ref{prop:group_action}, \ref{prop:gen_vector}, and \cite[Proposition 14]{harvey71}.

\begin{prop}
The equivalence classes of cyclic data sets of genus $g$ and degree $n$ are in one-to-one correspondence with conjugacy classes of finite cyclic subgroups of $\map(S_g)$ of order $n$ or conjugacy classes of $\z_n$-actions on $S_g$.
\end{prop}

\noindent The conditions of Definition \ref{defn:dataset} was first given by Harvey \cite[Theorem 4]{harvey66}. Harvey, in \cite[Proposition 14]{harvey71}, constructed explicit automorphisms of $\Gamma$ to establish Nielsen's result. The cyclic data set corresponding to a periodic mapping class $F \in \map(S_g)$ will be denoted by $D_F$.

\subsection{Birman-Hilden theory}
\label{subsec:bh_theory}
Let $F \in \map(S_g)$ be a periodic mapping class and $p:S_g\to \Orb_F\approx S_{g_0,k}$ be the corresponding finite-sheeted regular branched cover. The \textit{symmetric mapping class group}, denoted by $\smap_p(S_g)$, is the subgroup of $\map(S_g)$ consisting of mapping classes represented by homeomorphisms which preserve the fibers of $p$. The \textit{liftable mapping class group}, denoted by $\lmap_p(S_{g_0,k})$, is the subgroup of $\map(S_{g_0,k})$ consisting of mapping classes represented by homeomorphisms that lift under $p$. A cover is said to have the \textit{Birman-Hilden property} if two isotopic fiber-preserving homeomorphisms are fiber-isotopic. It is known that \cite{birman73,harvey75} the cover $p$ has the Birman-Hilden property. The Birman-Hilden property is equivalent to the existence of the following short exact sequence \cite{birman73,margalit21}
\begin{equation}
\label{eqn:bhseq}
1 \rightarrow \langle F \rangle \rightarrow \smap_p(S_g) \rightarrow \lmap_p(S_{g_0,k}) \rightarrow 1.
\end{equation}
Further, it is known \cite[Theorem 4]{birman73} that $\smap_p(S_g)$ is the normalizer of $F$ in $\map(S_g)$. 

Let $\Gamma(\Orb_F) = (g_0;n_1,n_2, \dots ,n_k)$ be the signature of $\Orb_F$. The subgroup of $\map(S_{g_0,k})$ consisting of mapping classes that preserve the orders of the marked points of $\Orb_F$ according to $\Gamma(\Orb_F)$ will be denoted by $\map_{\Gamma}(S_{g_0,k})$. The group $\out^+(\Gamma)$ is the image of the projection map $\aut(\Gamma)\to \out(\Gamma)$ restricted to $\aut^+(\Gamma)$, where $\out(\Gamma)$ is the group of outer automorphisms of $\Gamma$. There is an isomorphism (see \cite[Proposition 2.2]{broughton92}) between $\map_{\Gamma}(S_{g_0,k})$ and $\out^+(\Gamma)$. Consequently, throughout this article, $\map_{\Gamma}(S_{g_0,k})$ will be identified with $\out^+(\Gamma)$ and their elements will be used interchangeably.

\begin{rem}
For $F \in \map(S_g)$, the normalizer and centralizer of $F$ will be denoted by $N(F)$ and $C(F)$, respectively. The natural projection maps $\pi_1: \z_n^{\times}\times \aut^+(\Gamma)\to \z_n^{\times}$ and $\pi_2:\z_n^{\times}\times \aut^+(\Gamma)\to \out^+(\Gamma)$ will be used throughout the paper.
\end{rem}

\noindent The following result of Broughton (see \cite[Theorem 3.2]{broughton92}) gives an algebraic characterization of the Birman-Hilden property.

\begin{theorem}
\label{thm:norm_cent}
Let $F \in \map(S_g)$ be a periodic mapping class and $\Gamma_F$ be a generating $\Gamma$-vector corresponding to $\langle F \rangle$. Let $\map_{\Gamma_F}$ and $\z_n^{\times}(\Gamma_F)$ denote the group $\pi_2(\text{Stab}(\Gamma_F))$ and $\pi_1(\text{Stab}(\Gamma_F))$, where $\text{Stab}(\Gamma_F)$ is the stabilizer of $\Gamma_F$ under $\z_n^{\times} \times \aut^+(\Gamma)$-action. Then there are exact sequences
\begin{equation}
\label{eqn:neseqn}
1 \rightarrow \langle F \rangle \rightarrow N(F) \rightarrow \map_{\Gamma_F} \rightarrow 1
\end{equation}
and
\begin{equation}
\label{eqn:ceseqn}
1 \rightarrow C(F) \rightarrow N(F) \rightarrow \z_n^{\times}(\Gamma_F) \rightarrow 1.
\end{equation}
\end{theorem}

We conclude this subsection with the following useful remark.
\begin{rem}
\label{rem:iso_norm}
From the exact sequences (\ref{eqn:bhseq}) and (\ref{eqn:neseqn}), it follows that
\begin{equation}
\label{eqn:bh_alt}
\lmap_p(S_{g_0,k}) \cong N(F)/ \langle F \rangle \cong  \map_{\Gamma_F}.
\end{equation}
Hence, we have a description of liftable mapping classes as orientation-preserving outer automorphisms of $\Gamma$ stabilizing a fixed generating $\Gamma$-vector corresponding to $\langle F \rangle$-action on $S_g$ up to an automorphism of $\langle F \rangle$. Let $\clmap_p(S_{g_0,k})$ denote the image of $C(F)$ in $\lmap_p(S_{g_0,k})$. By Theorem \ref{thm:norm_cent} and Equation (\ref{eqn:bh_alt}), we have:
\begin{equation}
\label{eqn:lm/clm}
\begin{split}
\lmap_p(S_{g_0,k})/\clmap_p(S_{g_0,k})&\cong (N(F)/\langle F \rangle)/(C(F)/\langle F \rangle)\\
& \cong N(F)/C(F)\\
& \cong \z_n^{\times}(\Gamma
_F).
\end{split}
\end{equation}
\noindent Hence, the elements of $\clmap_p(S_{g_0,k})$ are orientation-preserving outer automorphisms of $\Gamma$ stabilizing a fixed generating $\Gamma$-vector corresponding to the $\langle F \rangle$-action on $S_g$ with $\aut(\langle F \rangle)$ acting trivially.
\end{rem}

\section{Generating the liftable mapping class group of spherical cyclic covers}
\label{sec:liftable}
In this section, we derive a generating set for liftable mapping class groups associated with spherical cyclic actions on surfaces. For a periodic mapping class $F\in \map(S_g)$, an $\langle F \rangle$-action on $S_g$ is called \textit{spherical} if $g_0(D_F)=0$. In this case, we say that $F$ is a \textit{spherical mapping class} and $S_g \to \Orb_F$ is a \textit{spherical cover}. Throughout this paper, we will assume that $F\in \map(S_g)$ is a spherical periodic mapping class.

\begin{rem}
Let $F$ be a spherical mapping class with the data set
\[
D_F=(n,0;(d_1,n_1),(d_2,n_2), \dots , (d_k,n_k)).
\]
We observe that the tuple
\[
\left(c_1=\frac{n}{n_1}d_1, c_2=\frac{n}{n_2}d_2, \dots , c_k=\frac{n}{n_k}d_k\right)
\]
is precisely a generating $\Gamma$-vector associated to $F$. This tuple will be denoted by $\Gamma_F$ and called the \textit{generating $\Gamma$-vector} associated with $D_F$.
\end{rem}

For the rest of this section, we fix a spherical mapping class $F\in \map(S_g)$ with associated data set $D_F$, generating $\Gamma$-vector $\Gamma_F$, and spherical cover $p:S_g\to S_{0,k}$. For some fixed $1 \leq j \leq k-1$ such that $|\gamma_j|=|\gamma_{j+1}|$, consider the map $\phi_j:\Gamma \to \Gamma$ defined on generators by:
\[
\phi_j(\gamma_j)=\gamma_{j+1}, \, \phi_j(\gamma_{j+1})=\gamma_{j+1}^{-1}\gamma_j\gamma_{j+1}, \text{ and } \phi_j(\gamma_i) = \gamma_i \text{ when } i \neq j.
\]
Since
\[
\phi_j\left(\prod_{i=1}^k \gamma_i\right)=\prod_{i=1}^k \gamma_i,
\]
by Theorem \ref{thm:aut_gamma}, $\phi_j$ is an orientation-preserving automorphism of $\Gamma$. At the level of generating $\Gamma$-vectors, $\phi_j$ acts by permuting $c_j$ and $c_{j+1}$. The following result (see \cite[Proposition 2.6]{broughton91}) describes the action of $\aut^+(\Gamma)$ on the generating $\Gamma$-vectors.

\begin{prop}
For $\z_n^{\times}\times \aut^+(\Gamma
)$-action on generating $\Gamma$-vectors, $\aut^+(\Gamma)$ acts on generating $\Gamma$-vectors by arbitrarily permuting the $c_j$'s while preserving their orders. Furthermore, $\z_n^{\times}$ acts on the generating $\Gamma$-vectors by multiplication.
\end{prop}

The action of $\map(S_{0,k})$ on $H_1(S_{0,k},\z)$ induces a surjective homomorphism $\psi: \map(S_{0,k})\to \Sigma_k$ with kernel, denoted by $\pmap(S_{0,k})$, known as the \textit{pure mapping class group} of $S_{0,k}$. Therefore, we have the following short exact sequence
\begin{equation}
1\rightarrow \mathrm{PMod}(S_{0,k}) \rightarrow \map(S_{0,k}) \xrightarrow{\psi} \Sigma_k \rightarrow 1.
\end{equation}
For $i=1,2,\dots ,k-1$, let $\sigma_i\in \map(S_{0,k})$ be a half-twist such that $\psi(\sigma_i)=(i,i+1)$. For $1\leq i<j<k$, let $a_{ij}=(\sigma_{j-1}\sigma_{j-2}\cdots \sigma_{i+1})\sigma_i^2(\sigma_{j-1}\sigma_{j-2}\cdots \sigma_{i+1})^{-1}$. It is known \cite{ghaswala17,birman74} that
\[
\map(S_{0,k})=\langle \sigma_i \mid 1\leq i<k \rangle
\text{~and~}
\pmap(S_{0,k})= \langle a_{ij} \mid 1\leq i <j <k \rangle.
\]

For simplicity, we will also denote the restriction of $\psi$ onto $\lmap_p(S_{0,k})$ and $\clmap(S_{0,k})$ by $\psi$. Since, elements of $\pmap(S_{0,k})$, by definition, act trivially on $H_1(S_{0,k},\z)$, it follows that $\pmap(S_{0,k})$ is a subgroup of $\lmap_p(S_{0,k})$. Since $\pmap(S_{0,k})$ fixes each branch point, it follows that $\pi_1 \circ \pi_2^{-1}(\mathrm{PMod}(S_{0,k}))$ is trivial. Therefore, by Remark \ref{rem:iso_norm}, $\pmap(S_{0,k})$ is also a subgroup of $ \clmap_p(S_{0,k}).$ Hence, we have the following short exact sequences,

\begin{equation}
\label{eqn:lmod}
1 \rightarrow \pmap(S_{0,k}) \rightarrow \lmap_p(S_{0,k}) \xrightarrow{\psi} H_N \rightarrow 1
\end{equation}
and
\begin{equation}
\label{eqn:clmod}
1 \rightarrow \pmap(S_{0,k}) \rightarrow \clmap_p(S_{0,k}) \xrightarrow{\psi} H_C \rightarrow 1,
\end{equation}
where $H_N$ and $H_C$ are subgroups of $\Sigma_k$. Throughout this paper, we will denote the images of $\lmap_p(S_{0,k})$ and $\clmap_p(S_{0,k})$ under $\psi$ by $H_N$ and $H_C$, respectively. By Equation (\ref{eqn:lm/clm}), we have the following short exact sequence
\begin{equation}
\label{eqn:lm/clm_se}
1 \to \clmap_p(S_{0,k}) \to \lmap_p(S_{0,k})\xrightarrow{q} \z_n^{\times}(\Gamma_F) \to 1.
\end{equation}

Let $B$ be the collection of all half-twists $\tau_{i,j}\in \map(S_{0,k})$ that permute the branch points corresponding to $c_i$ and $c_j$ in $\Gamma_F$ such that $n_i = n_j$ and $d_i = d_j$, where $i<j$. Since $(1,\tau_{i,j})\cdot \Gamma_F = \Gamma_F$, we have $\tau_{i,j} \in \clmap_p(S_{0,k})$. Now, we describe a generating set for $\lmap_p(S_{0,k})$.

\begin{theorem}
\label{thm:main}
For $A=\{a_{ij}\mid 1\leq i<j<k\}$ such that $\langle A \rangle=\pmap(S_{0,k})$ and $B$ as above, we have $\clmap_p(S_{0,k}) = \langle A \cup B \rangle$.
\end{theorem}
\begin{proof}
Given the short exact sequence (\ref{eqn:clmod}), it is enough to show that $\psi(B)$ generates $H_C$. For $\sigma \in H_C$, there exist a $\delta \in \clmap_p(S_{0,k})$ such that $\psi(\delta)=\sigma$. Without loss of generality, let $(1,2,\dots,s)$ be a cycle in the cycle decomposition of $\sigma$ for some $s\leq k$. Since $\delta \in \mathrm{CLMod}_p(S_{0,k})$, we have $(1,\delta)\cdot \Gamma_F=\Gamma_F$. Thus $c_1= c_j$ for every $1\leq j\leq s$. It follows that $n_1=n_j$ and $d_1 = d_j$ for all $1\leq j \leq s$. Since $(1,2,\dots,s)=(1,s)(1,s-1)\cdots (1,2)$ and $(1,2,\dots,s)$ was an arbitrary cycle in the decomposition of $\sigma$, we have that $\sigma$ is a product of the $\psi(\tau_{i,j})$ for $1\leq i<j \leq k$. Hence, $A\cup B$ generates $\clmap_p(S_{0,k})$.
\end{proof}

\noindent In the following remark, we describe an explicit construction of a finite set $C \subset \lmap_p(S_{0,k})$ such that $q(C)=\z_n^{\times}(\Gamma_F)$. 

\begin{rem}
\label{rem:main}
For $\ell \in \z_n^{\times}(\Gamma_F)$ such that $\ell \neq 1$, since $q$ is surjective, there exists a $\delta \in \lmap_p(S_{0,k})$ such that $q(\delta) = \ell$. Without loss of generality, for $s\leq k$, let $(1,2,\dots, s)$ be one of the disjoint cycles in the cycle decomposition of $\psi(\delta)$. Since $(\ell,\delta) \cdot \Gamma_F = \Gamma_F$, we have
\[
(\ell c_{s},\ell c_1, \dots ,\ell c_{s-1} ) = (c_1, c_2, \dots , c_{s}).
\]
Thus, we have $\ell^s \equiv 1 \pmod{n_1}$ and for $2 \leq i \leq s$, we have $c_i = \ell c_{i-1}$ and $c_1 = \ell c_s $. Hence,
\begin{equation}
\label{eqn:st_genvec}
(c_1, c_2, \dots , c_s) = (c_1, \ell c_1 , \dots , \ell^{s-1} c_1 ).
\end{equation}
For $1<j\leq s$, we observe that $(1,2,\dots,s)=(1,s)(1,s-1)\cdots (1,2)$ and
\begin{equation}
\label{eqn:arb_half_twist}
(1,j)=\psi(\sigma_1\sigma
_2\cdots \sigma_{j-2}\sigma_{j-1}\sigma_{j-2}\cdots \sigma_2 \sigma_1 ).
\end{equation}
By Equation (\ref{eqn:st_genvec}), it follows that $\Gamma_F$ is a disjoint union of tuples of the form $(c, \ell c , \dots , \ell^{s-1} c)$, where $c\in \z_n^{\times}$. Hence, by Equation (\ref{eqn:arb_half_twist}), there exists a $\delta_{\ell} \in \map(S_{0,k})$ such that $\psi(\delta_{\ell})=\psi(\delta)$. Since $\psi(\delta_{\ell})=\psi(\delta)$, we have $(\ell,\delta_{\ell})\cdot \Gamma_F =\Gamma_F$, and therefore, $\delta_{\ell}\in \lmap_p(S_{0,k})$ such that $q(\delta_{\ell})=\ell$. For $C=\{\delta_{\ell}\in \lmap_p(S_{0,k})\mid \ell \in \z_n^{\times}(\Gamma_F)\}$, by (\ref{eqn:lm/clm_se}), it follows that $\psi(B\cup C)$ generates $H_N$. 
\end{rem}

\noindent The following corollary is a direct consequence of Theorem \ref{thm:main}, Remark \ref{rem:main}, and Equation (\ref{eqn:lmod}).

\begin{cor}
\label{cor:main}
For $A$, $B$ defined as above, and  $C=\{\delta_{\ell}\in \lmap_p(S_{0,k})\mid \ell \in \z_n^{\times}(\Gamma_F)\}$, we have $\lmap_p(S_{0,k}) = \langle A \cup B \cup C \rangle$.
\end{cor}

\section{Applications}
\label{sec:app}
In this section, we derive several applications of Theorem \ref{thm:main}.
\subsection{An algorithm for deriving a presentation for the normalizer and the centralizer of a spherical mapping class}
We first recall a standard method to find a presentation for a group extension. Consider a short exact sequence of groups
\[
1 \rightarrow N \xrightarrow{i} G \xrightarrow{j} Q \rightarrow 1.
\]
Let $N = \langle S_1 \mid R_1 \rangle$ and $Q = \langle S_2 \mid R_2 \rangle$ be a presentation such that for $q=1,2$, $S_q$ and $R_q$, are the set of generators and relators for the corresponding groups, respectively. For a set $S$, the free group on $S$ will be denoted by $F(S)$. For each $s\in S_2$, consider an $s'\in j^{-1}(\{s\})$. Define
\[
S_{Q} := \{ s' \in G \mid  s \in S_2\} \text{ and } S_N := \{ i(a) \mid a \in S_1\}.
\]

Any $r \in R_2$ can be written as $r = s_1^{\epsilon_1}s_2^{\epsilon_2} \dots s_k^{\epsilon_k}$, where $s_q \in S_2$ and $\epsilon_q = \pm 1$ for $1 \leq q \leq k$. For $r' = (s_1')^{\epsilon_1}(s_2')^{\epsilon_2} \dots (s_k')^{\epsilon_k} \in G$, we have $j(r') = r.$ Since $r = 1 \in Q,$ we have $r' \in \ker j = \mathrm{Im}~i$. Thus, there exists $a_{r'}\in F(i(S_1))$ such that $i(a_{r'}) = r'$. Define
\[
R_Q := \{ i(a_{r'})=r' \mid j(r') = r \in R_2\}.
\]
Since $i(N)$ is normal in $G$, for every $s \in S_2$ and $a \in S_1$, we have $s'i(a)(s')^{-1} \in i(N)$. Let $s_a \in F(i(S_1))$ such that $s'i(a)(s')^{-1} = s_a$. Define
\[
R_{NQ} = \{s'i(a)(s')^{-1}=s_a \mid a \in S_1, ~s \in S_2\}
\]
and
\[
R_N = \{i(a_1)^{\epsilon_1}i(a_2)^{\epsilon_2} \cdots i(a_k)^{\epsilon_k}=1 \mid a_q \in S_1,~ a_1^{\epsilon_1}a_2^{\epsilon_2} \dots a_k^{\epsilon_k} \in R_1, ~\epsilon_q = \pm 1,~1\leq q\leq k \}.
\]

The following lemma is a well-known result in combinatorial group theory (see \cite[Proposition 2.55]{cgt_holt}).
\begin{lem}
\label{lem:presen}
For a short exact sequence of groups
\[
1 \rightarrow N \xrightarrow{i} G \xrightarrow{j} Q \rightarrow 1
\]
and given presentations $N = \langle S_1 \mid R_1 \rangle$, $Q= \langle S_2 \mid R_2 \rangle$, a presentation for $G$ is given by
\[
\langle S_N \cup S_Q \mid R_N \cup R_Q \cup R_{NQ} \rangle,
\]
where $S_N,~S_Q,~R_N,~R_Q$, and $R_{NQ}$ are defined above.
\end{lem}
\noindent The elements of $S_N$ and $S_Q$ will be called \textit{kernel generators} and \textit{lifted generators}, respectively. The equations of the sets $R_N$, $R_Q$, and $R_{NQ}$ will be called \textit{kernel relations}, \textit{lifted relations}, and \textit{conjugacy relations}, respectively.

We will now describe well-known \cite{birman74,ghaswala17} presentations for $\map(S_{0,k})$ and $\pmap(S_{0,k})$.

\begin{lem}
\label{lem:genset_mod0}
\begin{enumerate}[(a)]
\item The group $\map(S_{0,k})$ is generated by $\sigma_i$ for $1\leq i<k$ and has the following relations.
\begin{enumerate}[(i)]
\item $[\sigma_i,\sigma_j]=1$ for $|i-j|>1$.
\item $\sigma_i\sigma_{i+1}\sigma_i=\sigma_{i+1}\sigma_i\sigma_{i+1}$.
\item $(\sigma_1\sigma_2\cdots\sigma_{k-1})^k=1$.
\item $\sigma_1\sigma_2\cdots \sigma_{k-1}\sigma_{k-1}\cdots \sigma_2\sigma_1=1$.
\end{enumerate}
\item The group $\pmap(S_{0,k})$ is generated by $a_{ij}$ for $1\leq i<j<k$ and has the following relations.
\begin{enumerate}[(i)]
\item $[a_{pq},a_{rs}]=1$, where $p<q<r<s$.
\item $[a_{ps},a_{qr}]=1$, where $p<q<r<s$.
\item $[a_{rs}a_{pr}a_{rs}^{-1},a_{qs}]=1$, where $p<q<r<s$.
\item $a_{pr}a_{qr}a_{pq}=a_{qr}a_{pq}a_{pr}=a_{pq}a_{pr}a_{qr}$ for $p<q<r$.
\item $(a_{12}a_{13}\cdots a_{1(k-1)})\cdots (a_{(k-3)(k-2)}a_{(k-3)(k-1)})(a_{(k-2)(k-1)})=1$.
\end{enumerate}
\end{enumerate}
\end{lem}

Now, we provide the following algorithm for deriving presentations for $\lmap_p(S_{0,k})$, $C(F)$, and $N(F)$.

\begin{algo}
\label{algo:main}
Let $F \in \map(S_g)$ be a spherical mapping class with given cyclic data set
\[
D_F=(n,0;(d_1,n_1),(d_2,n_2), \dots, (d_k,n_k)).
\]
\begin{enumerate}[Step 1:]
\item From the given data set $D_F$, construct the set $B$ of all half-twists $\tau_{i,j}$ such that $\psi(\tau_{i,j})=(i,j)$, $n_i = n_j$, and $d_i = d_j$, where $i<j$.
\item By using Remark \ref{rem:main}, construct a finite set $C \subset \lmap_p(S_{0,k})$ such that $q(C)$ generates $\z_n^{\times}(\Gamma_F)$.
\item For $A=\{a_{i,j} \mid 1\leq i <j <k \}$, $A \cup B \cup C$ (resp. $A \cup B$) is a generating set for $\lmap_p(S_{0,k})$ (resp. $\clmap_p(S_{0,k})$) (follows by Theorem \ref{thm:main} and Corollary \ref{cor:main}).
\item Find a presentation for $\psi(\lmap_p(S_{0,k}))$ and $\psi(\clmap_p(S_{0,k}))$ in generating set $\psi(B \cup C)$ and $\psi(B)$, respectively (determination of such presentations will depend on $D_F$).
\label{step:pres_H1H2}
\item Use Lemmas \ref{lem:presen} - \ref{lem:genset_mod0} and the presentations obtained in Step \ref{step:pres_H1H2} to obtain a presentation for $\lmap_p(S_{0,k})$ and $\clmap_p(S_{0,k})$.
\label{step:presn}
\item Construct the sets $\tilde{A}$, $\tilde{B}$, and $\tilde{C}$ consisting of one lift for each element of $A$, $B$, and $C$, respectively. (For determining these lifts, the theory developed in \cite{rajeevsarathy19,kapdi23} can be used).
\item Use Lemma \ref{lem:presen} to obtain presentations of $N(F)$ (resp. $C(F)$) by lifting the presentations of $\lmap_p(S_{0,k})$ (resp. $\clmap_p(S_{0,k})$) obtained in Step \ref{step:presn} in the generating set $\{F\}\cup \tilde{A} \cup \tilde{B}\cup \tilde{C}$ (resp. $\{F\}\cup \tilde{A} \cup \tilde{B}$).
\end{enumerate}
\end{algo}

\subsection{Normalizers and centralizers of irreducible periodic mapping classes} 
A periodic mapping class $F\in \map(S_g)$ is called \textit{irreducible} if $\Gamma (\Orb_F)=(0,n_1,n_2,n_3)$, that is, $\Orb_F\approx S_{0,3}$. Thus, irreducible periodic mapping classes are spherical. Since $\Orb_F\approx S_{0,3}$, we have $\map(S_{0,3})\cong \Sigma_3$. Therefore, in this case, we identify $\map(S_{0,3})$ with $\Sigma_3$. The group with a presentation of the form
\[
\langle x,y\mid x^m=1,~y^n=1,~xyx^{-1}=y^{\ell}\rangle,
\]
where $\ell \in \z_n^{\times}$ with $\ell\neq 1$, will be denoted by $\z_n\rtimes_{\ell} \z_m$. The following result describes normalizers and centralizers of irreducible periodic mapping classes. We note that the result about the centralizer has also been obtained in earlier work (see \cite[Proposition 4.9]{kapdi24}).

\begin{prop}
\label{prop:norm_irr}
For $g\geq 2$, let $F\in \map(S_g)$ be an irreducible periodic mapping class of order $n$ with corresponding generating $\Gamma$-vector $\Gamma_F=(c_1,c_2,c_3)$ and signature $(0;n_1,n_2,n_3)$. Then the following statements hold.
\begin{enumerate}[(i)]
\item If there exists an $\ell\in \z_n^{\times}$ such that $\ell^3=1$ and $(c_1,c_2,c_3)=(c_1,\ell c_1, \ell^2 c_1)$, then $\ell\neq 1$, $\lmap_p(S_{0,3})\cong \z_3$, $C(F)=\langle F\rangle$, and $N(F)\cong \z_n \rtimes_{\ell}\z_3$.
\item If there exists an $\ell \in \z_n^{\times}$ such that $\ell^2=1$, $\ell\equiv 1 \pmod{n_1}$, and $(c_1,c_2,c_3)=(c_1,c_2,\ell c_2)$, then $\lmap_p(S_{0,3})\cong \z_2$. Furthermore:
\begin{enumerate}[(a)]
\item If $\ell=1$, then $n\leq 2g+2$ and $C(F)=N(F)\cong \z_n \times \z_2$.
\item If $\ell \neq 1$, then $C(F)=\langle F \rangle$ and $N(F)\cong \z_n \rtimes_{\ell} \z_2$.
\end{enumerate}
\item Otherwise, $\lmap_p(S_{0,3})=1$ and $C(F)=N(F)=\langle F \rangle$.
\end{enumerate}
\end{prop}

Before proving Proposition \ref{prop:norm_irr}, we provide the following example which shows that all of the cases mentioned in the statement of the proposition occur.

\begin{exmp}
In Table \ref{tab:first}, we list the isomorphism type of the normalizers and the centralizers of irreducible periodic mapping classes up to conjugacy and power in $\map(S_3)$, and the conjugacy classes of their generators. In this table, $D_G$ denotes the conjugacy class of the periodic map, which lifts under the cover $p:S_3 \to \Orb_F$. 

\begin{table}[ht]
\centering
\begin{tabular}{|c|c|c|c|c|}
\hline 
\rule[-1ex]{0pt}{2.5ex} \text{Sr. No.} & $D_F$ & $N(F)$ & $C(F)$ & $D_G$ \\ 
\hline 
\rule[-1ex]{0pt}{2.5ex} 1 & $(7,0;(1,7),(2,7),(4,7))$ & $\z_7\rtimes_2 \z_3$ & $\z_7$ & $(3,1;(1,3),(2,3))$ \\ 
\hline 
\rule[-1ex]{0pt}{2.5ex} 2 & $(7,0;(5,7),(1,7),(1,7))$ & $\z_7\times \z_2$ & $\z_7\times \z_2$ & $(2,0;(1,2)_8)$ \\ 
\hline 
\rule[-1ex]{0pt}{2.5ex} 3 & $(8,0;(1,4),(1,8),(5,8))$ & $\z_8\rtimes_5 \z_2$ & $\z_8$ & $(2,1;(1,2)_4)$ \\ 
\hline 
\rule[-1ex]{0pt}{2.5ex} 4 & $(8,0;(3,4),(1,8),(1,8))$ & $\z_8\times \z_2$ & $\z_8\times \z_2$ & $(2,0;(1,2)_8)$ \\ 
\hline 
\rule[-1ex]{0pt}{2.5ex} 5 & $(9,0;(1,3),(1,9),(5,9))$ & $\z_9$ & $\z_9$ & - \\ 
\hline 
\rule[-1ex]{0pt}{2.5ex} 6 & $(12,0;(1,2),(1,12),(5,12))$ & $\z_{12}\rtimes_5 \z_2$ & $\z_{12}$ & $(2,1;(1,2)_4)$ \\ 
\hline 
\rule[-1ex]{0pt}{2.5ex} 7 & $(12,0;(2,3),(1,4),(1,12))$ & $\z_{12}$ & $\z_{12}$ & - \\ 
\hline 
\rule[-1ex]{0pt}{2.5ex} 8 & $(14,0;(1,2),(3,7),(1,14))$ & $\z_{14}$ & $\z_{14}$ & - \\ 
\hline 
\end{tabular}
\caption{Isomorphism classes of the normalizers and the centralizers of irreducible periodic mapping classes up to conjugacy and power in $\map(S_3)$.}
\label{tab:first} 
\end{table}
\end{exmp} 

\begin{proof}[Proof of Proposition \ref{prop:norm_irr}]
For $\delta \in \map(S_{0,3})$ and $\ell\in \z_n^{\times}$, assume that $(\ell,\delta)\cdot \Gamma_F=\Gamma_F$.

\textbf{Case 1:} For $|\delta|=3$, without loss of generality, let $\delta=(1,2,3)$. Then we have
\[
(\ell,\delta)\cdot \Gamma_F=\left(\ell c_3,\ell c_1, \ell c_2\right)=(c_1,c_2,c_3).
\]
This implies that $n_1=n_2=n_3=n$, $c_2\equiv \ell c_1 \pmod n$, $c_3\equiv \ell c_2 \equiv \ell ^2 c_1 \pmod n$, $c_1 \equiv \ell c_3 \pmod n$, and $\ell^3=1$. From the Riemann-Hurwitz formula (\ref{eqn:rh}), we have $n=2g+1$. If $\ell=1$, then $3c_1\equiv 0 \pmod{2g+1}$. Since $\gcd(c_1,2g+1)=1$, we have $2g+1=3$, which is not possible as $g\geq 2$. Thus, $\ell \neq 1$. From Theorem \ref{thm:main} and Remark \ref{rem:main}, it follows that $\clmap_p(S_{0,3})=1$ and $\lmap_p(S_{0,3})\cong \z_3$. Hence $C(F)=\langle F \rangle$ and $N(F)\cong \z_n \rtimes_{\ell} \z_3.$

\textbf{Case 2:} For $|\delta|=2$, since $n_1\leq n_2\leq n_3$, without loss of generality, let $\delta=(2,3)$. Then we have 
\[
(\ell,\delta)\cdot \Gamma_F=\left(\ell c_1,\ell c_3,\ell c_2\right)=(c_1,c_2,c_3).
\]
This implies that $n_2=n_3=n$, $\ell\equiv 1 \pmod{n_1}$, $c_3\equiv \ell c_2 \pmod n$, $c_2\equiv \ell c_3 \pmod n$, and $\ell^2=1$. Thus, by Theorem \ref{thm:main} and Remark \ref{rem:main}, we have $\lmap_p(S_{0,3})\cong \z_2$. If $\ell=1$ then $\lmap_p(S_{0,3})=\clmap_p(S_{0,3})$, and therefore $C(F)=N(F)\cong \z_n \times \z_2$. By a result of Maclachlan \cite{4g+4}, the order of a finite abelian subgroup of $\map(S_g)$ can be at most $4g+4$. Since $\z_n\times \z_2 $ is abelian, we must have $n \leq 2g+2$. If $\ell \neq 1$, then $\clmap_p(S_{0,3})=1$. Hence $C(F)=\langle F \rangle$ and $N(F)\cong \z_n \rtimes_{\ell} \z_2$.

\textbf{Case 3:} Otherwise, we have $\lmap_p(S_{0,3})=1$, and so it follows that $C(F)=N(F)=\langle F \rangle$.
\end{proof}

\subsection{A presentation for the liftable mapping class group associated with some reducible periodic mapping classes}
Since we have already determined liftable mapping class groups, normalizers, and centralizers of all irreducible periodic mapping classes, we now focus on some reducible periodic mapping classes. Let $F\in \map(S_g)$ be an irreducible periodic mapping class with
\[
D_F=(n,0;(d_1,n_1),(d_2,n_2),(1,n)).
\]
From the theory developed in \cite{rajeevsarathy19}, $D_F$ can be realized by $2\pi/n$-rotation $\F$ of a polygon $\mathcal{P}$ with side-pairing. Without loss of generality, we assume that $F$ has such a representative. We note that the tuple $(1,n)$ in $D_F$ corresponds to the fixed point at the center of $\mathcal{P}$. Consider the maps $\F$ and $\F^{-1}$ on two distinct copies of $\mathcal{P}$. Since the sum of the angles at the center of the polygon $\mathcal{P}$ for $\F$ and $\F^{-1}$ adds up to $0 \pmod{2\pi}$, we remove small open (invariant) disks around the centers of two copies of $\mathcal{P}$ and glue the resultant boundaries. This leads to a map $\F'$ representing a spherical periodic mapping class $F'\in \map(S_{2g})$ with
\[
D_{F'}=(n,0;(d_1,n_1),(-d_1,n_1),(d_2,n_2),(-d_2,n_2)).
\]
In the following proposition, we derive a presentation for $\lmap_p(S_{0,4})$, where $p:S_{2g} \to \Orb_{F'}$.

\begin{prop}
\label{prop:present}
For $g$ even and $n>g+1$, let $F \in \map(S_g)$ be a periodic mapping class with
\[
D_F=(n,0;(d_1,n_1),(-d_1,n_1),(d_2,n_2),(-d_2,n_2)).
\]
Then the following statements hold.
\begin{enumerate}[(i)]
\item When $n_1=2$, we have:
\begin{enumerate}
\item $[N(F):C(F)]=2$ and $[\map(S_{0,4}):\lmap_p(S_{0,4})]=6$.
\item $\clmap_p(S_{0,4})=\left \langle \sigma_1, a_{13} \mid (\sigma_1 a_{13})^2=1 \right \rangle$.
\item $\lmap_p(S_{0,4})=\left \langle \sigma_1, \sigma_3, a_{13} \mid \sigma_3^2=\sigma_1^2,~[\sigma_1,\sigma_3]=(\sigma_1 a_{13})^2=(\sigma_3 a_{13})^2=1 \right \rangle$.
\end{enumerate}
\item When $n_1 \neq 2$, we have:
\begin{enumerate}
\item $[N(F):C(F)]=2$ and $[\map(S_{0,4}):\lmap_p(S_{0,4})]=12$.
\item $\clmap_p(S_{0,4})=\pmap(S_{0,4})$.
\item $\lmap_p(S_{0,4})=\left \langle a_{12}, a_{13}, \delta=\sigma_1^{-1}\sigma_3 \mid \delta^2=[a_{12},\delta]=[a_{13},\delta]=1\right \rangle$.
\end{enumerate}
\end{enumerate}
\end{prop}

\begin{proof}
\textbf{Case 1:} First, assume that $n_1=2$. Since $\lcm(n_1,n_2)=n$, either $n_2=n$ or $n_2=n/2$. By the Riemann-Hurwitz formula (\ref{eqn:rh}), it follows that either $\Gamma(\Orb_F)=(0;2,2,2g,2g)$ or $\Gamma(\Orb_F)=(0;2,2,g+1,g+1)$. Up to power, either
\begin{enumerate}[(i)]
\item $D_F=(2g+2,0;(1,2),(1,2),(1,g+1),(-1,g+1))$ or
\item $D_F=(2g,0;(1,2),(1,2),(1,2g),(-1,2g))$.
\end{enumerate}
By Theorem \ref{thm:main}, it follows that $H_C=\langle \psi(\sigma_1)=(1,2)\rangle \cong \z_2$. For $\delta=\sigma_3$, we have $(-1,\delta)\cdot \Gamma_F=\Gamma_F$. Therefore, $\delta \in \lmap_p(S_{0,4})$. Thus, by Remark \ref{rem:main}, we get $
H_N=\langle (1,2),(3,4)\rangle\cong K_4$, where $K_4$ is the Klein four-group. Hence,
\[
[N(F):C(F)]=[\lmap_p(S_{0,4}):\clmap_p(S_{0,4})]=[K_4,\z_2]=2
\]
and
\[
[\map(S_{0,4}):\lmap_p(S_{0,4})]=[\Sigma_4:K_4]=6.
\]

We now use Algorithm \ref{algo:main} to obtain a presentation for $\lmap_p(S_{0,4})$. Assume that
\[
H_N=\langle c=(1,2),d=(3,4)\mid c^2=d^2=[c,d]=1 \rangle \text{ and } H_C=\langle c \mid c^2=1\rangle.
\]
For generators of $\map(S_{0,4})$ and $\pmap(S_{0,4})$, we follow the notation as in Lemma \ref{lem:genset_mod0}. It can be seen that $\pmap(S_{0,4})$ is a free group of rank $2$ generated by $\{a_{12},a_{13}\}$. Therefore, the kernel generators are $S_A=\{a_{12},a_{13}\}$. The lifted generators are $S_C=\{\sigma_1,\sigma_3\}$. Since $a_{12}=\sigma_1^2$, we have $\lmap_p(S_{0,4})=\langle \sigma_1,\sigma_3,a_{13}\rangle$. By Lemma \ref{lem:presen}, there are the following relations in $\lmap_p(S_{0,4})$:
\begin{enumerate}[(i)]
\item Lifted relations:
\begin{enumerate}[(a)]
\item $\sigma_3^2=\sigma_1^{2}$.
\item $[\sigma_1,\sigma_3]=1$.
\end{enumerate}
\item Conjugacy relations:
\begin{enumerate}[(a)]
\item $\sigma_1a_{13}\sigma_1^{-1}=a_{13}^{-1}\sigma_1^{-2}$.
\item $\sigma_3a_{13}\sigma_3^{-1}=a_{13}^{-1}\sigma_3^{-2}$.
\end{enumerate}
\end{enumerate}
Hence, the above presentation simplifies as stated.

\textbf{Case 2:} We now assume that $n_1 \neq 2$. In this case $d_1 \neq -d_1$. Since $n>g+1$, $n_1 \neq n_2$. By Theorem \ref{thm:main}, we have $\clmap_p(S_{0,4})=\pmap(S_{0,4})$. Furthermore, for $\delta=\sigma_1^{-1}\sigma_3$, since $\psi(\delta)=(1,2)(3,4)$, we have $(-1,\delta)\cdot \Gamma_F=\Gamma_F$. Therefore $\delta\in \lmap_p(S_{0,4})$. In this case, we have  
\[
H_N=\langle c=(1,2)(3,4) \mid c^2=1\rangle.
\]
Thus, it follows that
\[
[N(F):C(F)]=[\lmap_p(S_{0,4}):\clmap_p(S_{0,4})]=[\z_2:1]=2
\]
and
\[
[\map(S_{0,4}):\lmap_p(S_{0,4})]=[\Sigma_4:\z_2]=12.
\]
We now use Algorithm \ref{algo:main} to derive a presentation for $\lmap_p(S_{0,4})$. The kernel generators are $S_A=\{a_{12},a_{13}\}$ and the lifted generators are $S_C=\{\delta=\sigma_1^{-1}\sigma_3\}$. Therefore, we have $\lmap_p(S_{0,4})=\langle a_{12},a_{13},\delta\rangle$. By Lemma \ref{lem:presen}, the following relations hold in $\lmap_p(S_{0,4})$:
\begin{enumerate}[(i)]
\item Lifted relations: $\delta^2=1$.
\item Conjugacy relations:
\begin{enumerate}
\item $\delta a_{12}\delta^{-1}=a_{12}$.
\item $\delta a_{13}\delta^{-1}=a_{13}$.
\end{enumerate}
\end{enumerate}
Hence, our assertion follows.
\end{proof}

Now, we derive a presentation for $N(F)$ and $C(F)$ for $F \in \map(S_g)$ with
\[
D_F=(2g+2,0;(1,2),(1,2),(1,g+1),(-1,g+1))
\]
by using Lemma \ref{lem:presen} and the presentations of $\clmap_p(S_{0,4})$ and $\lmap_p(S_{0,4})$ derived in Proposition \ref{prop:present}. By \cite{2g+2}, $F$ is a reducible mapping class of the highest order. Furthermore, it is known \cite{kulkarni97} that any periodic mapping class having a representative with a fixed point can be realized by a rotation of a polygon. Since $F$ has no fixed points, it can not be realized as a rotation of polygons. As described before, $F$ has been constructed by pasting two irreducible cyclic actions. This realization of $D_F$ will be used to derive a presentation for $N(F)$.

\begin{cor}
\label{cor:presen_2g+2}
For an even integer $g>0$, let $F\in \map(S_g)$ be such that
\[
D_F=(2g+2,0;(1,2),(1,2),(1,g+1),(-1,g+1)).
\]
Then
\[
C(F)=\left\langle F, G_1, G_2 \mid F^{2g+2}=[G_1,F]=[G_2,F]=1, ~(G_1G_2)^2=F^{g+2}\right\rangle
\]
and
\begin{align*}
\begin{split}
N(F)=\big\langle F, G_1, G_2, G_3 \mid & ~F^{2g+2}=[G_1,F]=[G_2,F]=[G_1,G_3]=1,\\ &~(G_1G_2)^2=F^{g+2}, ~G_3FG_3^{-1}=F^{-1}, G_1^2=G_3^2, ~(G_3G_2)^2=F^{g+1} \big\rangle.
\end{split}
\end{align*}
\end{cor}

\begin{proof}
We provide the computation details for $g=2$ as similar computations and methods will generally work. We will need the oriented simple closed close curves $ab,~bc, ~de,~ef$, and $c_1$ shown in Figure \ref{fig:hexagons} for describing the generators of the normalizer of $F$. (Here, $ab$ is the curve homotopic to the concatenation of $a$ and $b$.)
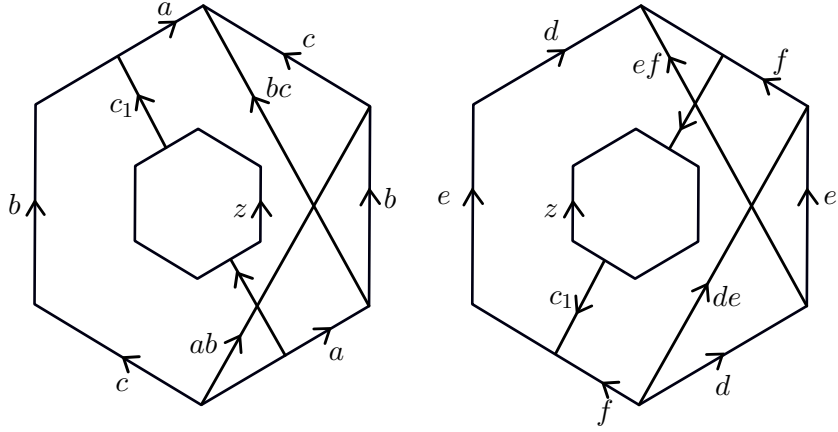
\begin{figure}[ht]
\centering
\labellist
\small
\pinlabel $a$ at 145, 25
\pinlabel $b$ at 170, 98
\pinlabel $c$ at 132, 172
\pinlabel $a$ at 65, 187
\pinlabel $b$ at -5, 95
\pinlabel $c$ at 45, 10
\pinlabel $c_1$ at 45, 140
\pinlabel $z$ at 100, 92
\pinlabel $ab$ at 83, 30
\pinlabel $bc$ at 118, 150

\pinlabel $d$ at 325, 10
\pinlabel $e$ at 375, 98
\pinlabel $f$ at 352, 162
\pinlabel $d$ at 245, 177
\pinlabel $e$ at 195, 98
\pinlabel $f$ at 270, -3
\pinlabel $c_1$ at 250, 50
\pinlabel $z$ at 245, 92
\pinlabel $de$ at 327, 52
\pinlabel $ef$ at 290, 160

\endlabellist
\includegraphics[scale=.8]{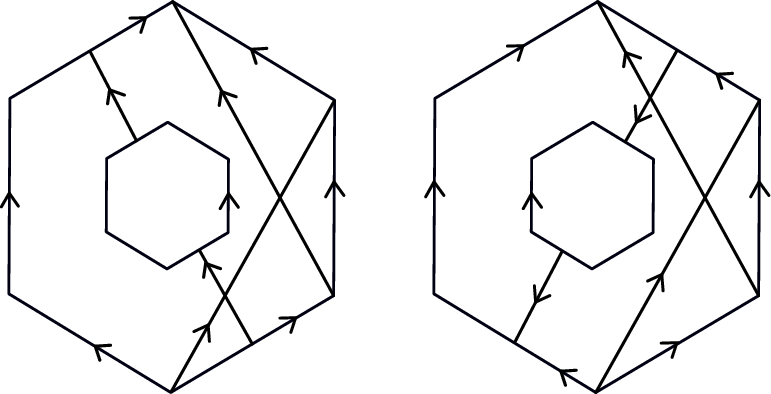}
\caption{A realization of a reducible periodic mapping class $F\in \map(S_2)$ of order $6$ with reduction system $\{z\}$.}
\label{fig:hexagons}
\end{figure}
From \cite[Example 4.5]{dhanwani23}, we get that $F$ is conjugate to $T_{ab}T_{bc}T_{ef}^{-1}T_{de}^{-1}$. Without loss of generality, we assume that $F=T_{ab}T_{bc}T_{ef}^{-1}T_{de}^{-1}$. For $c_2:=F(c_1)$ and $c_3:=F(c_2)$, let $G_1=T_{c_1}T_{c_2}T_{c_3}$, $G_2=T_{ab}T_{bc}$, and $G_3=G_1(T_{ab}T_{bc}T_{c_3}T_{de}T_{ef})^3$. By comparing actions of generators of $\lmap_p(S_{0,4})$ and their lifts on suitable curves, one can see that $\sigma_1$ lifts to $G_1$, $a_{13}$ lifts to $G_2$, and $\sigma_1^{-1}\sigma_3$ lifts to the handle swap map $G=(T_{ab}T_{bc}T_{c_3}T_{de}T_{ef})^3$ under $p$. Thus $\sigma_3$ lifts to $G_3=G_1G$. Now, we find relations in $N(F)$ and $C(F)$. According to Lemma \ref{lem:presen}, we have the following relations in $N(F)$:
\begin{enumerate}[(i)]
\item Kernel relations: $F^6=1$.
\item Conjugacy relations:
\begin{enumerate}[(a)]
\item $G_1FG_1^{-1}=F^{i_1}$.
\item $G_2FG_2^{-1}=F^{i_2}$.
\item $G_3FG_3^{-1}=F^{i_3}$.
\end{enumerate}
\item Lifted Relations:
\begin{enumerate}[(a)]
\item $G_1^2G_3^{-2}=F^{i_4}$.
\item $[G_1,G_3]=F^{i_5}$.
\item $(G_1G_2)^2=F^{i_6}$.
\item $(G_3G_2)^2=F^{i_7}$.
\end{enumerate}
\end{enumerate} 
We make use of the fact that the symplectic representation $\psi: \map(S_2)\to \text{Sp}(4,\z)$ is faithful when restricted to finite subgroups of $\map(S_2)$. Now we list the images of the generators of $N(F)$ under the map $\psi$ with respect to the basis $\{[ab],[bc],[de],[ef]\}$ of $H_1(S_2, \z)$, where $[ab]$ is the homology class of the curve $ab$. We have the following:
\[
\psi(F)=
\begin{pmatrix}
0 & -1 & 0 & 0 \\ 
1 & 1 & 0 & 0 \\ 
0 & 0 & 1 & 1 \\ 
0 & 0 & -1 & 0
\end{pmatrix},
\\
~\psi(G_1)=
\begin{pmatrix}
0 & -2 & -2 & -1 \\ 
2 & 2 & 1 & 2 \\ 
-2 & -1 & 0 & -2 \\ 
1 & 2 & 2 & 2
\end{pmatrix},
\\
~\psi(G_2)=
\begin{pmatrix}
0 & -1 & 0 & 0 \\ 
1 & 1 & 0 & 0 \\ 
0 & 0 & 1 & 0 \\ 
0 & 0 & 0 & 1
\end{pmatrix}
\]

\[
\psi(G)=
\begin{pmatrix}
0 & 0 & -1 & 0 \\ 
0 & 0 & 0 & -1 \\ 
-1 & 0 & 0 & 0 \\ 
0 & -1 & 0 & 0
\end{pmatrix},\text{ and }
\\
\psi(G_3)=
\begin{pmatrix}
2 & 1 & 0 & 2 \\ 
-1 & -2 & -2 & -2 \\ 
0 & 2 & 2 & 1 \\ 
-2 & -2 & -1 & -2
\end{pmatrix}.
\]
Now, it follows that $i_1=1$, $i_2=1$, $i_3=-1$, $i_4=1$, $i_5=1$, $i_6=4$, and $i_7=3$. Hence, our asserting follows. (Alternatively, writing all generators as a product of Dehn twists, one can directly perform computations to verify the relations in the presentation of $N(F)$.)
\end{proof}

The following result is a direct Corollary of Proposition \ref{prop:present}.

\begin{cor}
For integers $g>1$ and $g+1<n<2g$ such that $g$ is even, let $F \in \map(S_g)$ be a periodic mapping class with
\[
D_F=(n,0;(d_1,n_1),(-d_1,n_1),(d_2,n_2),(-d_2,n_2)).
\]
Then
\[
C(F)=\langle F,G_1,G_2 \mid F^n=[G_1,F]=[G_2,F]=1 \rangle.
\]
\end{cor}

\begin{proof}
The assertion follows directly from Lemma \ref{lem:presen} and Proposition \ref{prop:present} by observing that $\clmap_p(S_{0,4})$ is a free group of rank $2$.
\end{proof}

\subsection{Recovering generating sets for the liftable mapping class groups of the hyperelliptic cover and the balanced superelliptic cover}
\label{subsec:lift_bh}
In this subsection, we recover the generating sets of the liftable mapping class groups associated with the hyperelliptic cover first derived by Birman-Hilden \cite{birman71}, and for the balanced superelliptic cover obtained by Ghaswala-Winarski \cite{ghaswala17}.

An $F \in \map(S_g)$ is said to be a \textit{hyperelliptic involution} if it has a representative $\F \in \homeo(S_g)$ of order $2$ which has $2g+2$ fixed points. The data set and generating $\Gamma$-vector corresponding to a hyperelliptic involution is given by
\[
D_F = (2,0;\underbrace{(1,2), (1,2), \dots , (1,2)}_{(2g+2)\text{ - times}})
\]
and
\begin{equation}
\label{eqn:gen_vec_hi}
\Gamma_F = (c_1 = 1, c_2 = 1, \dots , c_{2g+2} = 1),
\end{equation}
respectively.
\begin{prop}[Birman-Hilden {\cite[Theorem 8]{birman71}}]
\label{prop:b-h}
For $g\geq 2$, let $F \in \map(S_g)$ be a hyperelliptic involution and let $p:S_g \to S_{0,2g+2}$ be its associated cover. Then
\[
\lmap_p(S_{0,2g+2})=\langle \sigma_i\mid 1\leq i<2g+2 \rangle=\map(S_{0,2g+2}).
\]
\end{prop}
\begin{proof}
Since $|F| = 2$, we have $N(F) = C(F)$. Therefore, $\lmap_p(S_{0,2g+2}) = \clmap_p(S_{0,2g+2})$. By Theorem \ref{thm:main} and Equation (\ref{eqn:gen_vec_hi}), it follows that
\[
H_N=H_C=\langle \psi(\tau_{i,j})=(i,j)\mid 1\leq i<j \leq 2g+2\rangle.
\]
Since
\[
\langle \psi(\tau_{i,j})=(i,j)\mid 1 \leq i<j\leq 2g+2\rangle=\langle \psi(\sigma_i)=(i,i+1)\mid 1\leq i<2g+2\rangle=\Sigma_{2g+2},
\]
we have $H_N=\Sigma_{2g+2}$. The result now follows from the observation that
\[
\langle a_{i,j}\mid 1\leq i<j<2g+2\rangle \subset \langle \sigma_i \mid 1\leq i<2g+2\rangle.
\]
\end{proof}

A periodic mapping class $F\in \map(S_g)$ with data set
\[
D_F = (n,0;\underbrace{(1,n), (-1,n), \dots , (1,n),(-1,n)}_{(k+1)\text{ - pairs}})
\]
is called a \textit{balanced superelliptic map}, where $k=g/(n-1)$. The generating $\Gamma$-vector corresponding to $D_F$ is given by
\begin{equation}
\label{eqn:gen_vec_bsc}
\Gamma_F = (c_1 = 1, c_2 = -1, \dots , c_{2k+1} = 1, c_{2k+2}=-1).
\end{equation}

\begin{prop}[Ghaswala-Winarski {\cite[Lemma 5.1]{ghaswala17}}]
\label{prop:g-w}
For $g\geq 2$, let $F\in \map(S_g)$ be a balanced superelliptic mapping class and let $p:S_g \to S_{0,2k+2}$ be its associated cover. Then
\[
\lmap_p(S_{0,2k+2})=\langle a_{ij},\sigma_{i'+1}\sigma_{i'}\sigma_{i'+1}^{-1},\sigma_1\sigma_3\cdots \sigma_{2k+1} \mid 1\leq i<j< 2k+2,~1\leq i'\leq 2k\rangle.
\]
\end{prop}

\begin{proof}
By Theorem \ref{thm:main} and Equation (\ref{eqn:gen_vec_bsc}), it follows that
\[
H_C=\langle \psi(\tau_{ij})=(i,j)\mid 1\leq i<j \leq 2k+2, ~i\equiv j \pmod 2\rangle=\langle (i,i+2)\mid 1\leq i \leq 2k \rangle.
\]
For $\delta \in \lmap_p(S_{0,2k+2})$ and $\ell\in \z_n^{\times}$, we have $(\ell,\delta)\cdot \Gamma_F=\Gamma_F$. Since
\[
(\ell,\delta)\cdot \Gamma_F=\delta\cdot(\ell, -\ell, \dots \ell, -\ell),
\]
by Equation (\ref{eqn:gen_vec_bsc}), we have $\ell=\pm 1$. By Equation (\ref{eqn:gen_vec_bsc}), it follows that there exist an order $2$ rotation $\beta\in \map(S_{0,2k+2})$ such that $\psi(\beta)=(1,2)(3,4)\cdots (2k+1,2k+2)$ and $(-1,\beta)\cdot \Gamma_F=\Gamma_F$. Therefore $\beta\in \lmap_p(S_{0, 2k+2})$ and $\langle q(\beta)\rangle\cong H_N/H_C\cong \z_n^{\times}(\Gamma_F)$. By Remark \ref{rem:main}, we have
\[
H_N=\langle \psi(\tau_{i,i+2}),\psi(\beta)\mid1\leq i\leq 2k\rangle.
\]
For $1\leq i\leq 2k$, we observe that $\psi(\sigma_{i+1}\sigma_i\sigma_{i+1}^{-1})=(i,i+2)$ and $\psi(\sigma_1\sigma_3\cdots \sigma_{2k+1})=\psi(\beta)$. The assertion now follows by Corollary \ref{cor:main}.
\end{proof}

As an immediate consequence of Theorem \ref{thm:norm_cent}, we provide an alternative proof of the following result by Ghaswala-Winarski \cite{winarski17} about characterizing spherical covers for which $\lmap_p(S_{0,k})=\map(S_{0,k})$.
\begin{prop}[Ghaswala-Winsarski{\cite[Theorem 1.3]{winarski17}}]
For $g\geq 2$, let $F \in \map(S_g)$ be a spherical mapping class with the generating $\Gamma$-vector $\Gamma_F=(c_1,c_2,\dots,c_k)$. 
Then $\lmap_p(S_{0,k})=\map(S_{0,k})$ if and only if $c_1=c_2=\dots =c_k$ and $k\equiv 0 \pmod n$.
\end{prop}
\begin{proof}
We have $\lmap_p(S_{0,k})=\map(S_{0,k})$ if and only if $\sigma_i\in \lmap_p(S_{0,k})$ for every $1\leq i<k$. By Theorem \ref{thm:norm_cent}, $\sigma_i\in \lmap_p(S_{0,k})$ if and only if $(1,\sigma_i)\cdot \Gamma_F=\Gamma_F$. It follows that $(1,\sigma_i)\cdot \Gamma_F=\Gamma_F$ if and only if $c_1=c_2=\dots=c_k$ and $k\equiv 0 \pmod n$. The conclusion $k\equiv 0 \pmod n$ follows by Equation \ref{defn:dataset}(\ref{eqn:angle_sum}).
\end{proof}

\subsection*{Acknowledgments} The first author was supported by the Prime Minister Research Fellowship (PMRF) Scheme instituted by the Ministry of Education, India.  The second author was supported by the ANRF MATRICS grant of the Govt. of India with grant number MTR/2023/000610. The third author acknowledges the financial support from the Swarna Jayanti Fellowship grant of Dr. Mahender Singh with grant number SB/SJF/2019-20/04.

\bibliographystyle{abbrv}
\bibliography{liftable.bib}
\end{document}